\documentclass[a4paper,12pt]{amsart}

\usepackage{amsmath,amssymb,amsthm,stmaryrd}
\usepackage{hyperref} 
\usepackage[T1]{fontenc}
\usepackage{latexsym}
\usepackage{enumerate}
\usepackage{fancyhdr}
\usepackage{latexsym}
\usepackage{dsfont}				
\usepackage{mathrsfs}			
\usepackage[active]{srcltx}		
\usepackage{esint}
\usepackage{pdfsync}

\usepackage{mathrsfs}
\usepackage{graphicx}

\usepackage[usenames,dvipsnames]{xcolor}

\usepackage[nospace,noadjust]{cite}

\newtheorem{Def}{Definition}[section]

\newtheorem{thm}[Def]{Theorem}

\numberwithin{equation}{section}

\newcommand{\R}{\ensuremath{\mathbb{R}}}

\newcommand{\Z}{\ensuremath{\mathbb{Z}}}

\def\XXint#1#2#3{{\setbox0=\hbox{$#1{#2#3}{\int}$}
     \vcenter{\hbox{$#2#3$}}\kern-.5\wd0}}

\usepackage{comment}

\usepackage{color}

\definecolor{gr}{rgb}   {0.,   0.8,   0. } 
\definecolor{bl}{rgb}   {0.,   0.5,   1. } 
\definecolor{mg}{rgb}   {0.7,  0.,    0.7}

\title[Wave packet function spaces]{Using wave packet decompositions to construct function spaces: a user guide}

\author{Pierre Portal}

\address{Pierre Portal, Mathematical Sciences Institute
and France-Australia Mathematical Sciences and Interactions
ANU-CNRS International Research Laboratory,
Australian National University,
Ngunnawal and Ngambri Country,
Canberra ACT 0200, Australia.}
\email{Pierre.Portal@anu.edu.au}

\begin{document}

\begin{abstract}      
We survey the construction of a range of function spaces used in harmonic analysis of PDE, including classical results as well as recent developments.
We frame these constructions in a common conceptual framework, where these function spaces arise as retracts of simple function spaces over phase space, through a projection associated with a wave packet decomposition.
Finding appropriate function spaces to study a given PDE then consists in choosing a relevant wave packet decomposition. We provide a user guide to making such choices, and constructing the corresponding function spaces. This is done mostly by surveying recent constructions, but we also include a new construction, adapted to Schr\"odinger operators of the form $\Delta - V$ for $V \geq 0$, as a sneak peek into upcoming joint work with Dorothee Frey, Andrew Morris, and Adam Sikora.\\
{\it MSC: 42B25, 42B37}.
\end{abstract}

\maketitle

\section{Introduction}
A wave packet decomposition of $L^{2}(\mathbb{R}^{d})$ is a decomposition of the identity of the form $ I=\mathcal{W}^{*} \mathcal{W}$, where 
$$
\mathcal{W}: L^{2} \left(\mathbb{R}^{d}\right) \rightarrow L^{2}\left(\mathbb{R}^{d} \times \mathbb{R}^{d}\right) 
$$
is a lift from the physical space $\mathbb{R}^{d}$ to the phase space $\mathbb{R}^{d}\times \mathbb{R}^{d}$. Heuristically, one can think of $f$ as a distribution of particles, and of $\mathcal{W}(f)(.,\xi)$ as the distribution of those particles that have momentum $\xi$. A simple example is given by the Littlewood-Paley lifting defined by
$$
\mathcal{W} f(x, \xi):=\psi(|\xi D|) f(x) \quad \forall x, \xi \in \mathbb{R}^{d},
$$
for some $\psi \in C_{c}^{\infty}\left(\mathbb{R}_{+}\right)$, with $\psi(|\xi D|)$ defined by $\widehat{\psi(|\xi D|) f}(\eta):=\psi(|\xi \eta|) \widehat{f}(\eta) \;\forall \xi, \eta \in \mathbb{R}^{d}$, 
where $\widehat{f}$ denotes the Fourier transform of $f$.
Given that $\mathcal{W} f$ is radial in the momentum variable $\xi$, and that $|\xi|$ is used as a dilation parameter, it is natural to consider $\mathcal{W}$ (abusing notation slightly) as the map
$$
\mathcal{W}: L^{2} \left(\mathbb{R}^{d}\right) \rightarrow L^{2}\left(\mathbb{R}^{d} \times(0, \infty) ; d x \frac{d \sigma}{\sigma}\right)
$$ 
defined by $\mathcal{W}(f)(x, \sigma):=\psi(\sigma|D|) f(x) \quad \forall x\in \mathbb{R}^{d} \quad \forall \sigma>0$.
Then, assuming the normalisation condition $\int_{0}^{\infty} \psi(\sigma)^{2} \frac{d \sigma}{\sigma}=1$, we have, as required, that
$$ \mathcal{W}^{*} \mathcal{W}(f)=\int_{0}^{\infty} \psi(\sigma|D|)^{2} f \frac{d \sigma}{\sigma}=f \quad \forall f\in L^{2}(\mathbb{R}^{d}).$$
We can then define function spaces over $\mathbb{R}^{d}$ through the norm $\|\mathcal{W}(f)\|_{Y}$ for any choice of function space $Y$ over $\mathbb{R}^{d}\times (0,\infty)$. For instance, choosing
$Y = L^{p}\left(\mathbb{R}^{d} ; L^{2} ((0,+\infty), \frac{d \sigma}{\sigma^{1+2 \alpha}})\right)$ gives an equivalent norm for the Sobolev Space $W^{\alpha, p}$ for $p \in (1,\infty)$ and $\alpha \in \mathbb{R}$:
$$
\|f\|_{W^{\alpha, p}} \sim\left\|\mathcal{W}(f)\right\|_{L^{p}\left(\mathbb{R}^{d} ; L^{2} ((0,+\infty), \frac{d \sigma}{\sigma^{1+2 \alpha}})\right)},
$$
while the same choice for $p=1$ and $\alpha=0$ gives an equivalent norm for the Hardy space $H^{1}(\mathbb{R}^{d})$.
More generally, the Triebel-Lizorkin space $F^{p,q}_{\alpha}$ corresponds to the choice $Y= L^{p}\left(\mathbb{R}^{d} ; L^{q} ((0,+\infty), \frac{d \sigma}{\sigma^{1+q \alpha}})\right)$, and the Besov space $B^{p,q}_{\alpha}$ corresponds to the choice $Y= L^{q}\left((0,+\infty), \frac{d \sigma}{\sigma^{1+q \alpha}});L^{p}(\mathbb{R}^{d})\right)$; see the textbook \cite{sawano}.
Many properties of these function spaces (e.g. interpolation and duality) can be deduced from the fact that they are retracts of $Y$, thanks to the key fact that 
$$
\mathcal{W}\mathcal{W}^{*} \in B(Y);
$$
see \cite{triebel}.
These classical function spaces play a critical role in harmonic analysis and its applications to PDE, because they are invariant under the action of large classes of pseudo-differential operators. There are, however, other crucial operators in harmonic analysis that do not act boundedly on these function spaces, except in the Hilbert space case where $p=q=2$. This is the case, for instance, of $\exp(i\Delta)$ and $\exp(i\sqrt{-\Delta})$; see \cite[Theorems 3.9.4 and 8.3.13]{abhn}. Even for heat semigroups $\exp(-L)$, when $L$ is a divergence form elliptic operator with sufficiently rough complex coefficients, boundedness is sometimes restricted to an interval of values of $p,q$ around $2$; see \cite{a}.\\

One can understand this phenomenon as a lack of adequation between the operators being studied, and the choice of wave packet decomposition used to analyse them (and hence to define relevant function spaces). While the action of pseudo-differential operators can typically be understood from their actions on Littlewood-Paley wave packets, the same cannot be said for operators such as $\exp(i\Delta)$.

 In this paper, we survey other choices of wave packet decomposition that give scales of function spaces well suited to analysing, respectively, $\exp(i\Delta)$ (Section \ref{sec:modulation}), $\exp(i\sqrt{-\Delta})$ (Section \ref{sec:HpFIO}), $\exp(-L)$ for $L$ a rough divergence form elliptic operator (Section \ref{sec:HardyOp}), and $\exp(-L)$ for $L$ the Ornstein-Uhlenbeck operator acting on $L^{2}(\mathbb{R}^{d}; d\gamma)$ where $\gamma$ is the Gaussian measure (Section \ref{sec:HardyOU}). In the final Section \ref{sec:new}, we prove a first result about a new wave packet decomposition that will feature in an upcoming work of the author with Dorothee Frey, Andrew Morris, and Adam Sikora. This wave packet decomposition is well suited to analysing Schr\"odinger operators of the form $\Delta-V$ for a non-negative potential $V$ in the reverse H\"older class $RH_{q}$ (for $q>d/2$).

The title of each section describes the region of phase-space where the wave packets are approximately localised. In this spirit, the Littlewood-Paley wave packet decompositions considered in the introduction could be described as ``localisation: energy levels in momentum".

For each choice, the questions that need to be answered are the following.
\begin{enumerate}
\item
{\bf Question 1.} Is $\mathcal{W}\mathcal{W}^{*}$ bounded on $Y$?
\item 
{\bf Question 2.} Which operators leave the spaces defined as a completion with respect to the norm $\|\mathcal{W}(f)\|_{Y}$ invariant?
\item
{\bf Question 3.} Do these spaces embed into known function spaces?
\end{enumerate}
A positive answer to Question 1 gives interpolation and duality results for the new spaces, via the abstract theory of retracts of Banach spaces. In Question 2, one wants to include solution maps for linear PDE, as well as multiplication operators. This makes the new spaces valid choices to run fixed point arguments in non-linear PDE. Finally, answers to Question 3 allows one to compare the use of these new spaces to more classical choices. This justifies constructing such spaces when they offer sharper results than known optimal results in classical spaces. 

\subsection*{Acknowledgments.}
The paper touches on multiple theories, each with a rich history and extensive literature. The references to these theories given here are by no means exhaustive, or even indicative of the most important results. They are merely chosen to guide the reader towards the literature in a way that fits the narrative of this article.\\

The paper was written as an extension of a talk given at the 2025 RIMS Symposium on Harmonic Analysis and Nonlinear Partial Differential Equations, and has been highly influenced by the symposium discussions. 
I am particularly grateful to Professor Neal Bez and Professor Yutaka Terasawa for their hospitality, and for sharing their mathematical insights, as well as to Professor Yohei Tsutsui for the great organisation of the symposium.\\

{\it AI statement.} ChatGPT 5.6 was used to identify remaining misprints in the final stages of proofreading, between acceptance and publication.

\section{Localisation: balls of fixed radius in momentum.}
\label{sec:modulation}
In this section, we survey the construction of {\bf modulation spaces}, as presented in the book \cite{bo}. These spaces are particularly well suited to the study of operators such as $\exp(i\Delta)$. Given $\psi \in C_{c}^{\infty}$, we consider $\mathcal{W} : L^{2}(\mathbb{R}^{d}) \to L^{2}(\mathbb{R}^{d}\times \mathbb{R}^{d})$ defined, for $f \in L^{2}(\mathbb{R}^{d})$ by
$$
\mathcal{W}f(x,\eta):=\psi (D+\eta)f(x) \quad \forall x,\eta \in \mathbb{R}^{d},
$$
where $\widehat{\psi (D+\eta)f}(\xi) = {\psi}(\xi+\eta)\widehat{f}(\xi)
\quad \forall \xi,\eta \in \mathbb{R}^{d}$. With the normalisation $\| \psi \|_{2} =1$, this defines a wave packet decomposition since
$$
\mathcal{W}^{*}\mathcal{W}f = \int  _{\mathbb{R}^{d}}  \psi (D+\eta)^{2}f d\eta = f.
$$
One then defines, for $p,q \in [1,\infty)$ and $s \in \mathbb{R}$, 
$$
\|f\|_{M^{p,q}_{s}} := \|\mathcal{W}f\|_{L^{q}(\mathbb{R}^{d},(1+|\eta|)^{sq}d\eta;L^{p}(\mathbb{R}^{d},dx))} \quad \forall f \in C_{c}^{\infty}(\mathbb{R}^{d}).
$$
Answering the questions laid out in the introduction, we have that
$$
\mathcal{W}\mathcal{W}^{*} \in B(L^{q}(\mathbb{R}^{d},(1+|\eta|)^{sq}d\eta;L^{p}(\mathbb{R}^{d},dx)))
$$
by a Schur's estimate (since $\psi(D+\eta)\psi(D+\xi)=0$ unless $|\eta-\xi| \lesssim 1$). We also have that a large class of Fourier Integral Operators (FIO) are bounded on $M^{p,q}_{s}$ as proven in \cite{cnr}. This includes $\exp(i\Delta)$ and $\exp(i\sqrt{-\Delta})$, and many other Fourier multipliers. There are even Fourier multipliers that are bounded on $M^{p,q}_{0}$ for all $1<p,q<\infty$ yet only bounded in $L^{p}$ for $p=2$; see \cite{bggo}. For multiplication operators, H\"older type inequalities hold; see \cite{t}. Modulation spaces are thus highly effective in the study of nonlinear Schr\"odinger equations; see for instance \cite{robert-modulation,bc,chkp}.
Turning to Question 3, we have embeddings such as
$$
B^{p,q}_{s} \subset M^{p,q}_{0} \subset B^{p,q}_{\overline{s}}
$$
for $s \geq d(max(0,1/q-min(1/p,1/p')))$ and $\overline{s} \leq d(min(0,1/q-max(1/p,1/p')))$ established in \cite{t}, and proven to be optimal in \cite{st}.
Using these embeddings, and the invariance of modulation spaces under the action of $\exp(i\Delta)$, one recovers the optimal mapping properties of this operator between Sobolev spaces:
$$
\exp(i\Delta) \in B(W^{s,p},W^{-s,p}) \quad \forall p\in(1,\infty) \quad
\forall s>\frac{d}{2}|1/p-1/2|.
$$
We also obtain the same result (and no better; see \cite{ks}) for $\exp(i\sqrt{-\Delta})$. However, the optimal mapping properties of $\exp(i\sqrt{-\Delta})$, as established in \cite{p} and \cite{m}, are:
$$
\exp(i\sqrt{-\Delta}) \in B(W^{s,p},W^{-s,p}) \quad \forall p\in(1,\infty) \quad
\forall s>\frac{(d-1)}{2}|1/p-1/2|.
$$
This suggests that modulation spaces are perfectly suited to studying Schr\"odinger equations, but not optimal for studying wave equations. A heuristic for this phenomenon is the fact that the phase of the symbol $\exp(i|\xi|^{2})$ can be effectively linearised on balls of fixed radius, whereas the phase of the symbol $\exp(i|\xi|)$ cannot. To study wave equations, one should thus use a wave packet decomposition that localises on regions where $\exp(i|\xi|)$ can be effectively replaced by $\exp(i\ell(\xi))$ for some linear map $\ell$. This is the subject of the next section.

\section{Localisation: directional slices of dyadic annuli in momentum.}
\label{sec:HpFIO}
In this section, we survey the construction of {\bf function spaces for decoupling}, introduced in \cite{hpr,hpry}. These spaces are particularly well suited to the study of operators such as $\exp(i\sqrt{-\Delta})$. 
Let $\psi \in C_{c} ^{\infty}$ be a radial function supported in $\{\xi \in \mathbb{R}^{d}\;;\; 1\leq |\xi| \leq 2\}$. We need to refine the Littlewood-Paley wave packet decomposition in an appropriate direction dependent manner. To do so, let us consider the momentum part of phase space in polar coordinates $(\omega,\sigma^{-1}) \in S^{d-1} \times (0,\infty)$, and define, for all $\omega \in S^{d-1}$, a function $\varphi_{\omega} \in C^{\infty}(\mathbb{R}^{d})$ supported in \begin{equation*}
\{\xi\in\mathbb{R}^{d} \;;\;|\xi| \geq 1/4,\; |\frac{\xi}{|\xi|}-\omega| \leq |\xi|^{-1/2}\}.
\end{equation*}
Note that, for a fixed $\sigma \in (0,1)$, $\psi_{\omega,\sigma}:\xi \mapsto \psi (\sigma \xi) \varphi_{\omega}(\xi)$ is supported in a region where $|\frac{\xi}{|\xi|}-\omega|  \lesssim \sigma^{1/2}$. This means that, as $\sigma$ approaches $0$ (i.e. in the high energy limit), $\xi$ is well approximated by the linearisation $(\xi \cdot \omega)\omega$ on the support of $\psi_{\omega,\sigma}$. We define accordingly, for $f \in L^{2}(\mathbb{R}^{d})$,
$$
\mathcal{W}f(x,\omega,\sigma):=
1_{(0,1)}(\sigma) \psi_{\sigma,\omega}(D)f(x) + 
1_{[1,e]}(\sigma)\rho_{0}(D)f(x) \quad \forall x \in \mathbb{R}^{d} \; \forall \omega \in S^{d-1} \; \forall \sigma>0,
$$
where \begin{equation*}
\rho_{0}:=\Big(1-\int_{0}^{1}\int_{S^{d-1}}\psi_{\omega,\sigma}^{2}d \omega \frac{d \sigma}{\sigma}\Big)^{1/2}.
\end{equation*}
Assuming further regularity and normalisation properties (see \cite[Section 3.1]{hpry}), we have that $\mathcal{W}^{*}\mathcal{W} = I$ (see \cite[Lemma 4.3]{hpr}), defining a wave packet decomposition.
We then define function spaces via the norm
$$
\|f\|_{\mathcal{L}^{q,p}_{\mathcal{W},s}}:= \|(I-\Delta)^{s/2}\mathcal{W}f\|_{L^{q}_{\omega}\left(L^{p}_{x}(L^{2}_{\sigma})\right)}, 
$$
for $1<p,q<\infty$ and $s\in \mathbb{R}$, denoting by $L^{q}_{\omega}\left(L^{p}_{x}(L^{2}_{\sigma})\right)$ the space
$
L^{q}(S^{d-1}; L^{p}(\mathbb{R}^{d};L^{2}((0,\infty),\frac{d\sigma}{\sigma})))
$.
Answering Question 1 from the introduction, the key fact that
$$
\mathcal{W}\mathcal{W}^{*} \in B(L^{q}_{\omega}\left(L^{p}_{x}(L^{2}_{\sigma})\right))
$$
is proven in \cite[Theorem 3.4]{hpry}. Turning to Question 2, a very large class of FIO is bounded on $\mathcal{L}^{p,p}_{\mathcal{W},0}$, as proven in \cite[Theorem 6.10]{hpr} (where this space is denoted by $\mathcal{H}^{p}_{FIO}$). The quintessential example $\exp(i\sqrt{-\Delta})$ is still bounded on $\mathcal{L}^{q,p}_{\mathcal{W},s}$ for $q\neq p$ as proven in \cite[Theorem 5.1]{hpry}. General FIO, however, are not bounded on $\mathcal{L}^{q,p}_{\mathcal{W},s}$ unless $p=q$. This is proven in \cite[Section 5.2]{hpry}. The boundedness of multiplication operators on $\mathcal{L}^{p,p}_{\mathcal{W},0}$ is considered in \cite{r}. 
These spaces can then be used to solve non-linear wave equations with rough coefficients, as shown in \cite{hr,rs}.
To conclude this section, let us point out answers to Question 3 given in \cite[Theorem 6.3]{hpry}. For $p\in (1,\infty)$ and $q \leq max(p,2)$, we have that:
$$
W^{s+s(p),p} \subset \mathcal{L}^{q,p}_{\mathcal{W},s}, $$
where $s(p) = (\frac{d-1}{2})|1/p-1/2|$.
If $q \geq min(p,2)$, we have the opposite inclusion:
$$
\mathcal{L}^{q,p}_{\mathcal{W},s} \subset W^{s-s(p),p}.$$
Combining these embeddings with the boundedness of $\exp(i\sqrt{-\Delta})$ on $\mathcal{L}^{q,p}_{\mathcal{W},s}$, we recover the optimal mapping properties of this operator:
$$
\exp(i\sqrt{-\Delta}) \in B(W^{s,p},W^{-s,p}) \quad \forall p\in(1,\infty) \quad
\forall s>\frac{(d-1)}{2}|1/p-1/2|.
$$
This is why $\mathcal{L}^{q,p}_{\mathcal{W},s}$ spaces are more suited to wave equations than modulation spaces are. 
The other interest of these spaces, especially when $q\neq p$, is that they encode quantities appearing in decoupling inequalities. For instance, \cite[Corollary 7.3]{hpry}, gives the following embedding
$$
\mathcal{L}^{2,p}_{\mathcal{W},s} \subset 
\left\{f\in W^{s-s(p)+1/p-\varepsilon}\;;\; supp \widehat{f} \subset \{\xi \in \mathbb{R}^{d}\;;\; R-1 \leq |\xi| \leq R+1\}\right\}
$$
for $R \geq 2$, $p \geq \frac{2(d+1)}{d-1}$, thanks to the $\ell^{2}$ decoupling inequality for the sphere proven in \cite{bd}. See \cite[Section 7]{hpry} for more results of this nature, and \cite{rs} for their use in proving local smoothing estimates for wave equations.

\section{Localisation: adapted energy levels in momentum.}
\label{sec:HardyOp}
In this section, we survey the construction of {\bf Hardy spaces associated with operators}, first introduced in the unpublished manuscript \cite{adm}, and then studied in many articles, starting with \cite{dy,hm,amr}. In previous sections, we have considered wave packet decompositions defined using Fourier multipliers. The corresponding function spaces turned out to be well suited to studying operators that are themselves not too far from Fourier multipliers, such as Calder\'on-Zygmund operators or FIO. Some heat semigroups $\exp(L)$, however, have kernels that lack the necessary smoothness or decay to ensure that $\exp(L)$ is a Calder\'on-Zygmund operator. This is the case, for instance, for uniformly elliptic divergence form operators with $L^{\infty}$ complex coefficients, for which harmonic analysis beyond  Calder\'on-Zygmund theory can nonetheless be developed; see \cite{a}.
In this theory, even the analogue $\nabla L^{-1/2}$ of the Hilbert transform can fail to be bounded on $L^p$ for some $p \in (1,\infty)$. Its $L^2$ boundedness is the celebrated solution of the Kato square root conjecture, established in \cite{ahlmt}. 

To construct appropriate function spaces, the idea is to replace localisation operators given by Fourier multipliers $\psi(D)$ by operators of the form $\psi(L)$. This can be done under minimal functional calculus assumptions on $L$ acting on $L^2$. For simplicity, we assume here that $L$ is non-negative self-adjoint on $L^2$, and take $\psi \in C_{c}^{\infty}(\mathbb{R}_{+})$, so that $\psi(L)$ can be defined using the spectral theorem.
 We then define, for $f \in L^{2}(\mathbb{R}^{d})$,
$$
\mathcal{W}f(x,\sigma):= \psi(\sigma^{2} L)f(x) \quad \forall x \in \mathbb{R}^{d} \quad \forall \sigma>0.
$$
Fixing the normalisation $\int _{0} ^{\infty} \psi(\sigma^{2})^{2} \frac{d \sigma}{\sigma} = 1$, this gives a wave packet decomposition since $\mathcal{W}^{*}\mathcal{W}f = \left(\int _{0} ^{\infty} \psi(\sigma^{2} L)^{2} \frac{d \sigma}{\sigma} \right) f = f$.

We could then consider the same classes of function spaces over phase space as before, namely, $L^{p}(\mathbb{R}^{d};L^{2}((0,\infty),\frac{d \sigma}{\sigma}))$. This would give rise to the Triebel-Lizorkin spaces associated with $L$ studied in \cite{kp}. These spaces are useful in a wide range of problems, but not when the operator $\exp(L)$ fails to be bounded on $L^p$ (as can be the case for some of the divergence form operators with complex coefficients mentioned above). To be able to treat these operators, one uses tent spaces $T^{p,2}$ instead of $L^{p}(\mathbb{R}^{d};L^{2}((0,\infty),\frac{d \sigma}{\sigma}))$. These spaces, introduced in \cite{cms} are defined via the norm
$$
\|F\|_{T^{p,2}}:=
\left( \int _{\mathbb{R}^{d}} \left( \int _{0} ^{\infty} \sigma^{-d} \int _{B(x,\sigma)} |F(y,\sigma)|^{2} dy \frac{d\sigma}{\sigma} \right)^{p/2} dx \right)^{1/p} 
$$
for all $F \in C_{c}^{\infty}(\mathbb{R}^{d} \times (0,\infty))$.
Like $L^{p}(\mathbb{R}^{d};L^{2}((0,\infty),\frac{d \sigma}{\sigma}))$, they measure the size of functions on phase space (that are radial in their momentum component) in an $L^p$ manner in position and an $L^2$ manner in momentum. The difference is that the function is not considered pointwise in position, but averaged over the ball where it is mostly located according to the uncertainty principle. Heuristically, this compensates for potentially pathological behaviour in the $x$ variable (as exhibited, for instance, by solutions of parabolic problems without a Harnack principle).
We now define, for $f \in L^{2}(\mathbb{R}^{d})$ belonging to the range of $L$, and $p \in [1,\infty)$:
$$
\|f\|_{H^{p}_{L}}:= \|\mathcal{W}f\|_{T^{p,2}}.
$$
Note that, as essentially shown in \cite{cms}, $H^{p}_{\Delta} = L^{p}$ for $p\in (1,\infty)$, and $H^{1}_{\Delta} = H^{1}$.\\
To answer the questions from the introduction, $\exp(L)$ needs to satisfy some off-diagonal decay assumptions, which generalise the kernel conditions of Calder\'on-Zygmund theory. Let us use the most common form of such assumptions given (with some constant $c>0$) by
\begin{equation}
\label{eq:GD}
\|1_{E} \exp(\sigma L) (1_{F}f)\|_{2} \lesssim \exp(-c \frac{d(E,F)^{2}}{\sigma}) \|1_{F}f\|_{2} \quad \forall f\in L^{2}(\mathbb{R}^{d}) 
\end{equation}
for all Borel sets $E,F\subset \mathbb{R}^{d}$ and all $\sigma>0$.
Answering Question 1, we have that $\mathcal{W}\mathcal{W}^{*} \in B(T^{p,2})$ for all $p \in [1,\infty)$; see, in particular, the proofs in \cite[Section 7]{hnp} and \cite{amm}.
The fundamental answer to Question 2 is that $L$ has a bounded holomorphic functional calculus on $H^{p}_{L}$; see for instance \cite{hnp,dl}. In many situations, Riesz transforms $\nabla L^{-1/2}$ also map $H^{p}_{L}$ to $L^{p}$; see for instance \cite{hm}. The typical answer to Question 3 is that $H^{p}_{L} = L^{p}$ for a range of values of $p$ including $2$, or at least, there is a norm equivalence on a subspace of $L^p$. Subtle forms of such results are given, for instance, in \cite[Chapter 6]{a}, \cite{fmp}, and \cite{as}.

\section{Localisation: position dependent critical energy levels in momentum.}
\label{sec:HardyOU}
In this section, we survey the construction of {\bf spaces over $\mathbb{R}^{d}$ equipped with the Gaussian measure} $d\gamma(x) = \exp(-|x|^{2})dx$. This is part of the field of Gaussian harmonic analysis, presented in the book \cite{u}.
In this field, the Laplacian operator is replaced by the Ornstein-Uhlenbeck operator $L=\frac{1}{2}\Delta -x \cdot \nabla$ that arises as the self-adjoint operator associated with the energy form
$$
f\mapsto \int _{\mathbb{R}^{d}} |\nabla f(x)|^{2} d\gamma(x).
$$
This operator plays a critical role in many areas, including mathematical physics (where one identifies $L^{2}(\gamma)$ with the Fock space, and $L$ with the observable that counts the number of particles in the system), and stochastic analysis (where $(\exp(tL))_{t \geq 0}$ is the transition semigroup of the stochastic differential equation $dX_{t} = -X_{t}dt + dW_{t}$). Its harmonic analysis is difficult for two reasons: the heat kernel of $\exp(tL)$ is not a Calder\'on-Zygmund kernel, and the measure $\gamma$ is not doubling. To overcome the first difficulty, one can take the same approach as in Section \ref{sec:HardyOp}, and use the functional calculus of $L$ instead of Fourier multipliers. This is particularly natural here because $\gamma$ is not invariant by translation, and using wave packet decompositions based on translation invariant operators would thus be unnatural. Despite the second difficulty, this approach
is good enough to solve various problems, and, in particular, to construct and use Besov or Triebel-Lizorkin spaces adapted to $\gamma$ and $L$; see \cite[Chapter 7]{u}. In some problems, however, the impact of the non-doubling nature of the Gaussian measure is too important, and a more refined wave packet decomposition is needed. This is the case, in particular, at the $p=1$ endpoint. We review here the construction of the corresponding Hardy space $h^{1}(\gamma)$ from \cite{po}.

The additional idea that is required is that the Gaussian measure has a doubling property when restricted to balls $B(x,a\rho(x))$ for a constant $a>0$ and $\rho:x\mapsto min(1,\frac{1}{|x|})$. A form of off-diagonal decay similar to \eqref{eq:GD} holds, provided one restricts the property to such balls, and the inverse energy parameter $\sigma$ to $[0,\rho(x)]$; see \cite[Lemma 3.1]{p}. This leads to the following definition, for $f \in L^{2}(\mathbb{R}^{d};\gamma)$,
$$
\mathcal{W}f(x,\sigma):= \left(1_{[0,\rho(x)]}(\sigma)
\psi(\sigma^{2} L)f(x),\int \limits _{\rho(x)} ^{\infty} \psi(\tau^{2} L)^{2}f(x) \frac{d \tau}{\tau} \right),
$$
as a map from $L^{2}(\mathbb{R}^{d},d\gamma)$ to $L^{2}(\mathbb{R}^{d}\times (0,\infty),d\gamma  \frac{d\sigma}{\sigma})\oplus L^{2}(\mathbb{R}^{d},d\gamma)$, using a function
$\psi$ of the form $z\mapsto Cz^{\frac{N+1}{2}}\exp(-\frac{1+a^{2}}{\alpha}z)$ for appropriate choices of constants $C,a,\alpha,N>0$; see \cite[Lemma 2.1]{p}. Picking $C>0$ appropriately gives a wave packet decomposition since
$$
\mathcal{W}^{*}\mathcal{W} f = \int_{0} ^{\infty} \psi(\sigma^{2}L)^{2} f \frac{d\sigma}{\sigma} = f,
$$
for $f\in L^{2}(\mathbb{R}^{d},d\gamma)$ in the range of $L$.
For such $f$, one then defines 
$$
\|f\|_{h^{p}(\gamma)}:= \|\mathcal{W}f\|_{t^{p,2}(\gamma)\oplus L^{p}(\gamma)},
$$
for $p\in [1,\infty)$, where $t^{p,2}(\gamma)$ denotes the local tent space with norm
$$
\|F\|_{t^{p,2}(\gamma)}:=
\left( \int _{\mathbb{R}^{d}} \left( \int _{0} ^{1} |\gamma(B(x,\sigma))|^{-1} \int _{B(x,\sigma)} |F(y,\sigma)|^{2} d\gamma(y) \frac{d\sigma}{\sigma} \right)^{p/2} d\gamma(x) \right)^{1/p}.$$
Answering Question 1 from the introduction, the fact that $\mathcal{W}\mathcal{W}^{*} \in B(t^{p,2}(\gamma)\oplus L^{p}(\gamma))$ is proven as part of the proof of \cite[Theorem 1.1]{p}. 
We then obtain, as an answer to Question 2, that the Riesz transforms $\partial_{j} L^{-1/2}$ map $h^{p}(\gamma)$ to $L^{p}(\gamma)$; see \cite[Section 6]{p}. Turning to Question 3, we have that $h^{p}(\gamma)=L^{p}(\gamma)$ for $p\in(1,\infty)$; see \cite[Section 5.3]{u}.
The $p=1$ theory from \cite{p} thus extends the celebrated $L^{p}$ boundedness result (for $p \in (1,\infty)$) of Gaussian Riesz transforms in Malliavin calculus. Note that these results could be improved by using a better Gaussian analogue of tent spaces to evaluate $\mathcal{W}f$, as recently constructed in \cite{fsu}.

One of the heuristics that one can take away from the Gaussian theory is that, when doing harmonic analysis of a problem driven by a hamiltonian that includes both kinetic and potential energy, it is useful to localise wave packets to regions of phase space where the potential energy is almost constant (here the balls of the form $B(x,\rho(x))$), and the kinetic energy is of the appropriate magnitude from an uncertainty principle point of view (meaning here that $\sigma \in [0,\rho(x)]$). In the next section, we bring this idea to the study of Schr\"odinger operators, including the (unitary equivalent) companion of the Ornstein-Uhlenbeck operator $\Delta -|x|^{2}$ acting on $L^{2}(\mathbb{R}^{d})$.

\section{Localisation: critical ball in position, dual energy level in momentum.}
\label{sec:new}
Let $d>2$. In this section, we construct a  {\bf new wave packet decomposition adapted to a Schr\"odinger operator} \(L=\Delta-V\) where $V$ denotes the multiplication operator by a function $V \in RH_{q}$ for some $q>\frac{d}{2}$, where
$RH_{q}$ denotes the set
$$\left\{v \in L_{loc}^{1}\left(\mathbb{R}^{d}\right) ; \exists C>0 \; \forall B(x, r) \subset \mathbb{R}^{d}\; \left(r^{-d}\int_{B(x, r)} V(y)^{q} d y\right)^{\frac{1}{q}} \leq Cr^{-d}\int_{B(x, r)} V\right\}.$$
 The critical radius associated with $V$, first systematically considered in \cite{s}, is defined by
$$\rho(x):=\sup \left\{r>0 ; \quad \frac{1}{r^{d-2}} \int_{B(x, r)} V(y) d y \leq 1\right\}.$$
We assume that $\rho \in L^{\infty}(\mathbb{R}^{d})$, and denote its supremum by $R$. This is the case, for instance, if $V=1$ (where $\rho(x)\sim1$) or $V(.)=|.|^{2}$
(where $\rho(x)\sim min(1,\frac{1}{|x|})$ as in Section \ref{sec:HardyOU}).
Cubes centred at $x \in \mathbb{R}^{d}$ and with side length  $\ell(Q) \in [1/2\rho(x),2\rho(x)]$ are called critical. The potential $V$ is essentially constant on such cubes.
One can extract from the standard dyadic cubes of $\mathbb{R}^{d}$ a partition $\mathbb{R}^{d} = \cup _{j\in \mathbb{Z}^{d}}Q_{j}$ by critical cubes with centres $(x_{j})_{j\in\mathbb{Z}^{d}}$.

Moreover, by \cite{dz}, we have that there exists $N_{\rho}\in \mathbb{N}$ such that:

\begin{equation}
\label{eq:overlap}
\forall j,k \in \mathbb{Z}^{d} \quad
|j-k| \geqslant N_{\rho}  \Rightarrow  1_{Q_{k}} 1_{2Q_{j}}=0.
\end{equation}

We will use the multiplication operators by $1_{Q_{j}}$ to localise our wave packets in position. For the momentum localisation, we pick
$\psi \in \mathcal{S}(\R)$ real and even, with $\psi(0)=0$, and such that
\(\widehat{\psi} \in C_{c}^{\infty}((0,+\infty))\). We then consider
\[\psi(L):=\frac{1}{\sqrt{2 \pi}} \int_{\mathbb{R}} \widehat{\psi}(\xi) \exp (i \xi \sqrt{-L}) d \xi.\]

By the finite speed of propagation of $(\cos(t\sqrt{-L}))_{t \in \mathbb{R}}$ (see \cite{remling})
and the finite overlap property \eqref{eq:overlap}, we have that:
\begin{equation}
\label{eq:finitespeed}
\exists N_{\rho} \in \mathbb{N} \quad \forall j,k \in \mathbb{Z}^{d} \quad \forall \sigma \in \left(0, \rho\left(x_{j}\right)\right) \quad
|j-k| \geqslant N_{\rho}  \Rightarrow  1_{Q_{k}} \psi\left(\sigma^{2} L\right)1_{Q_{j}}=0.
\end{equation}
Fixing the normalisation $\int_{0}^{\infty} \psi(\sigma^{2})^{2} \frac{d \sigma}{\sigma}=1$, we have the following reproducing formula for all $f \in L^{2}\left(\mathbb{R}^{d}\right)$:
\begin{align*}
 f&=\sum_{j \in \mathbb{Z}^{2}} 1_{Q_{j}}\int_{0} ^{\infty} \psi\left(\sigma^{2} L\right) \psi\left(\sigma^{2} L\right) (1_{Q_{j}}f) \frac{d \sigma}{\sigma}\\
&=\sum_{j \in \mathbb{Z}^{2}} \left[\int_{0}^{\rho\left(x_{j}\right)}  1_{Q_{j}}\psi\left(\sigma^{2} L\right) \psi\left(\sigma^{2} L\right) (1_{Q_{j}}f)\frac{d \sigma}{\sigma}+R_{j} f\right],
\end{align*} 
where, for all $j \in \mathbb{Z}^{d}$,
$$
R_{j} = 1_{Q_{j}} \left[\int_{\rho\left(x_{j}\right)} ^{\infty} \psi\left(\sigma^{2} L\right) \psi\left(\sigma^{2} L\right)\frac{d \sigma}{\sigma}\right]1_{Q_{j}},
$$
defines a positive operator on $L^{2}(\R^{d})$, since
\begin{align*}
\left\langle R_{j} f, f\right\rangle
= \int_{\rho\left(x_{j}\right)}^{\infty} \|\psi\left(\sigma^{2}L\right) 1_{Q_{j}}f\|_{2}^{2} \frac{d \sigma}{\sigma}.
\end{align*}
We can thus define:
\[\mathcal{W}: L^{2}\left(\mathbb{R}^{d}\right) \rightarrow \ell^{2}\left(\mathbb{Z}^{2} ; L^{2}\left(\mathbb{R}^{d} \times(0, \infty), d x \frac{d \sigma}{\sigma}\right)\right)\]
by
\[\mathcal{W}f(j, x, \sigma):=1_{\left[0, \rho\left(x_{j}\right)\right]}(\sigma)  \psi\left(\sigma^{2} L\right)1_{Q_{j}} f+1_{[R, 2 R]}(\sigma) \ln (2)^{-\frac{1}{2}} R_{j}^{\frac{1}{2}} f\]
for all $ j \in \mathbb{Z}^{d}$, $x\in \mathbb{R}^{d}$ and $\sigma>0$.

We then have a well defined wave packet decomposition since, for all $f\in L^{2}(\mathbb{R}^{d})$, we have that:
\[\mathcal{W}^{*}\mathcal{W}(f)=\sum_{j \in \mathbb{Z}^{2}}\left[\int_{0}^{\rho\left(x_{j}\right)} 1_{Q_{j}}\psi\left(\sigma^{2} L\right) \psi\left(\sigma^{2} L\right) 1_{Q_{j}}f \frac{d \sigma}{\sigma}+R_{j} f\right]=f.\]

The point of this wave packet decomposition is that $R_{j}$ will be a hypercontractive error term, while $1_{Q_{j}}-R_{j}$ will localise in position to a cube where the potential is almost constant. In the meantime, because $\psi \in \mathcal{S}(\mathbb{R}^{d})$, we will have an approximate spectral localisation around the points in the spectrum of $L$ of magnitude roughly $\sigma^{-2}$. In upcoming work with Dorothee Frey, Andrew Morris and Adam Sikora, we shall demonstrate the usefulness of such localisations in PDE applications.
Note, for now, that localisation on balls where the potential is constant coupled with wave packet methods has recently allowed Robert Schippa to prove local smoothing estimates for the wave equation associated with $\Delta-|.|^{2}$ (see \cite{robert}). In the meantime, harmonic analysis results for $\Delta-V$ are known to be better in the case $\rho \neq 0$ than in the case $\rho=0$ of the Laplacian. See, for instance \cite{dz} and \cite{b}. The wave packet decomposition studied here is designed to capture both phenomena.

In the present paper, we initiate the study of this wave packet decomposition by proving the following preliminary (but crucial) result.

\begin{thm}
Assume that, for all $\sigma>0$, $\psi(\sigma^{2}L)$ has an integral kernel
$K_{\sigma^{2}}$ such that: $\exists g, h \in L^{1}((0,+\infty),r^{d-1}dr) \quad \forall x, y \in \mathbb{R}^{d} \quad \forall j,k \in \mathbb{Z}^{d} \quad \forall \sigma \geqslant \rho\left(x_{j}\right)$
\begin{align}
\label{eq:kernel}
\left|1_{Q_{k}}(x) K_{\sigma^{2}}(x, y) 1_{Q_{j}}(y)\right| \leq \sigma^{-d} g\left(\frac{|x-y|}{\sigma}\right) h\left(\frac{\sigma}{\rho\left(x_{j}\right)}\right).
\end{align}
Then \(\mathcal{W}\mathcal{W}^{*}\) extends to a bounded linear operator on \(\ell^{p}\left(\mathbb{Z}^{d} ; L^{p}\left(\left(\mathbb{R}^{d} ; L^{2}\left((0,+\infty), \frac{d \sigma}{\sigma}\right)\right)\right)\right.\) for all \(p \in(1, \infty)\).
\end{thm}

Note that the assumption \eqref{eq:kernel} is satisfied by simple potentials such as $V=1$ or $V=|.|^{2}$. We will study, in future work, the class of potentials $V$ for which \eqref{eq:kernel} holds. Results such as \cite{k} and \cite{mp} suggest that it should hold for all $V \in RH_{q}$.
\begin{proof}
We assume, without loss of generality (by duality), that $p \in (1,2]$.
We denote the space $\ell^{p}\left(\mathbb{Z}^{d} ; L^{p}\left(\left(\mathbb{R}^{d} ; L^{2}\left((0,+\infty), \frac{d \sigma}{\sigma}\right)\right)\right)\right)$ by $\ell^{p}(L^{p}(L^{2}))$.\\
Let $F \in \ell^{p}(L^{p}(L^{2})) \cap \ell^{2}(L^{2}(L^{2}))$, and write 
$$
\mathcal{W}\mathcal{W}^{*}F={G}+{H}_{1}+{H}_{2}+H_{3},
$$
where, for all $j \in \mathbb{Z}^{d}$ and all  $\sigma>0$
\begin{align*}
G(j, \cdot, \sigma)&:=1_{\left[0,\rho\left(x_{j}\right)\right]}(\sigma) \psi\left(\sigma^{2} L\right) 1_{Q_{j}} \left(\int_{0}^{\rho(x_j)} \psi\left(\tau^{2} L\right) F(j, \cdot, \tau) \frac{d \tau}{\tau}\right),
\\
H_{1}(j, \cdot, \sigma)&:=1_{[R, 2 R]}(\sigma)(\ln 2)^{-\frac{1}{2}} R_{j}^{\frac{1}{2}}\left(\sum_{k \in \mathbb{Z}^{d}} \int_{0}^{\rho\left(x_{k}\right)} 1_{Q_{k}}\psi\left(\tau^{2} L\right)  F(k, \cdot, \tau) \frac{d \tau}{\tau}\right),\\
H_{2}(j, \cdot, \sigma)&:=1_{[R, 2R]}(\sigma)(\ln 2)^{-1} R_{j}^{\frac{1}{2}}\left(\sum_{k \in \mathbb{Z}^{d}} \int_{R}^{2 R} R_{k}^{\frac{1}{2}} F(k, \cdot, \tau) \frac{d \tau}{\tau}\right), \\
H_{3} F(j, \cdot, \sigma)&:=1_{\left[0, \rho\left(x_{j}\right)\right]}(\sigma) \psi\left(\sigma^{2} L\right) 1_{Q_{j}} (\ln 2)^{-1/2} \left(\sum_{k \in \mathbb{Z}^{d}} \int_{R}^{2 R} R_{k}^{\frac{1}{2}} F(k, \cdot, \tau) \frac{d \tau}{\tau}\right). 
\end{align*}

To estimate each of these terms, we use the following key facts.
\begin{align}
\label{eq:holocalc}
\left\|\left(\int_{0}^{\infty}\left|\psi\left(\tau^{2} L\right)  f \right|^{2}\right)^{\frac{1}{2}}\frac{d \tau}{\tau}\right\|_{L^{p}} &\sim \|f\|_{p} \quad \forall f \in L^{p}(\mathbb{R}^{d}),\\
\left\|\left(\int_{0}^{\infty}\left|\psi\left(\tau^{2} L\right)  F(\tau,.) \right|^{2}\right)^{\frac{1}{2}}\frac{d \tau}{\tau}\right\|_{L^{p}} &\sim \|F\|_{L^p(L^2)} \quad \forall F \in L^{p}(\mathbb{R}^{d};L^{2}((0,\infty);\frac{d\tau}{\tau})).
\end{align}
This is a consequence of the fact that $L$ has a $\gamma$ bounded H\"ormander functional calculus \cite{dos}, and the square function estimates \cite[Theorem 4.9]{kw}. Moreover, we have that
\begin{equation}
\label{eq:rbdSchur}
\left\|\left(\left.\int_{0}^{\infty}\left|\int_{0}^{\infty} \psi\left(\sigma^{2} L\right) \psi\left(\tau^{2} L\right) f \frac{d \tau}{\tau}\right|^{2} \frac{d \sigma}{\sigma} \right)^{\frac{1}{2}} \right. \right\|_{L^{p}}^{p} \lesssim
\left\|\left(\int_{0}^{\infty}\left|\psi\left(\tau^{2} L\right)  f \frac{d \tau}{\tau}\right|^{2}\right)^{\frac{1}{2}}\right\|_{L^{p}}
\end{equation}
for all $f \in L^{p}(\mathbb{R}^{d})$.
Given that $L$ has a $\gamma$-bounded H\"ormander functional calculus (see \cite[Theorem 1.2]{dkk}), and using the fact that 
$$
\psi\left(\sigma^{2} L\right) \psi\left(\tau^{2} L\right) = \frac{\sigma^{2}}{\tau^{2}} \overline{\psi}\left(\sigma^{2} L\right) \underline{\psi}\left(\tau^{2} L\right) = \frac{\tau^{2}}{\sigma^{2}} \underline{\psi}\left(\sigma^{2} L\right) \overline{\psi}\left(\tau^{2} L\right) \quad \forall \sigma,\tau>0,
$$
for some functions $\overline{\psi},\underline{\psi} \in \mathcal{S}(\mathbb{R}^{d})$, \eqref{eq:rbdSchur} is a consequence of the $\gamma$-bounded version of Schur's lemma proven in \cite[Proposition 4.2]{hmp}.
Similarly, we have that, for $j \in \mathbb{Z}^{d}$, 
\begin{equation}
\label{eq:rbdSchur2}
\left\|\left(\left.\int_{0}^{\infty}\left|\int_{0}^{\infty} \psi\left(\sigma^{2} L\right)1_{Q_{j}} \psi\left(\tau^{2} L\right) f \right|^{2}\frac{d \sigma d\tau}{\sigma \tau} \right)^{\frac{1}{2}} \right. \right\|_{L^{p}}^{p} \lesssim
\left\|\left(\int_{0}^{\infty}\left|1_{Q_{j}}\psi\left(\tau^{2} L\right)  f\right|^{2} \frac{d \tau}{\tau}\right)^{\frac{1}{2}}\right\|_{L^{p}}\end{equation}
for all $f \in L^{p}(\mathbb{R}^{d})$
using the fact that,  for $\sigma,\tau, \theta>0$, 
\begin{align*}
\psi\left(\sigma^{2} L\right) 1_{Q_{j}}\psi\left(\tau^{2} L\right) & = (\frac{\sigma^{2}}{\tau^{2}})^{\theta} \overline{\psi}\left(\sigma^{2} L\right) 
L^{\theta}1_{Q_{j}}L^{-\theta}
\underline{\psi}\left(\tau^{2} L\right) \\ & = (\frac{\tau^{2}}{\sigma^{2}})^{\theta} \underline{\psi}\left(\sigma^{2} L\right)
L^{-\theta}1_{Q_{j}}L^{\theta}
 \overline{\psi}\left(\tau^{2} L\right),
\end{align*}
and $L^{-\theta}1_{Q_{j}}L^{\theta}$, as well as $L^{\theta}1_{Q_{j}}L^{-\theta}$ are bounded on $L^{p}$ (uniformly in $j$) for $\theta$ small enough.
For $V=0$, this last fact is proven in \cite{runst-sickel}. The result for $V \in RH_{q}$ and $\theta$ small enough follows from
$$
\|L^{\theta} f\|_{p} \sim \|\Delta^{\theta} f\|_{p} + \|V^{\theta} f\|_{p} \quad \forall f \in D(L^{\theta})
$$ 
which, in turn, is a consequence of the results in \cite{ab}.

We estimate $\|G\|_{\ell^{p}\left(L^{p}\left(L^{2}\right)\right)}$ first, using 
\eqref{eq:rbdSchur2},  followed by \eqref{eq:holocalc}.
\[
\begin{aligned}
&\|G\|_{\ell^{p}\left(L^{p}\left(L^{2}\right)\right)} ^{p} \\ & \quad \lesssim \sum_{j \in \mathbb{Z}^{d}}  \left\|\left(\left.\int_{0}^{\infty}\left|\int_{0}^{\infty} \psi\left(\sigma^{2} L\right) 1_{Q_{j}}\psi\left(\tau^{2} L\right) 1_{\left[0, \rho\left(x_{j}\right)\right]}(\tau)  F(j, \cdot, \tau) \frac{d \tau}{\tau}\right|^{2} \frac{d \sigma}{\sigma} \right)^{\frac{1}{2}} \right. \right\|_{L^{p}}^{p} \\
& \quad \lesssim \sum_{k \in \Z^{d}}\left\|\left(\int_{0}^{\infty}\left|1_{\left[0, \rho\left(x_{k}\right)\right]}(\tau) 1_{Q_{k}}\psi\left(\tau^{2} L\right)  F(k, \cdot, \tau) \frac{d \tau}{\tau}\right|^{2}\right)^{\frac{1}{2}}\right\|_{L^{p}}^{p} \\
& \quad \lesssim \|F\|_{\ell^{p}\left(L^{p}\left(L^{2}\right)\right)} ^{p}.
\end{aligned}
\]
\newpage
Turning to the terms $H_{1},H_{2},H_{3}$, the crucial extra remark is that
\begin{equation}
\label{eq:localrest}
R_{j}^{\frac{1}{2}} = 1_{Q_{j}} R_{j}^{\frac{1}{2}}1_{Q_{j}}  \quad \forall j \in \mathbb{Z}^{d}.
\end{equation}
Indeed, for $k \neq j$, we have that $1_{Q_{k}} R_{j}^{\frac{1}{2}}  = R_{j}^{\frac{1}{2}}1_{Q_{k}} = 0$ since, for all $f \in L^{2}(\mathbb{R}^{d})$, we have the following:
$$
\left\|R_{j}^{\frac{1}{2}} 1_{Q_{k}} f\right\|_{2}^{2}=\left\langle R_{j} 1_{Q_{k}} f, 1_{Q_{k}} f\right\rangle =\int_{\rho\left(x_{j}\right)}^{\infty} \langle 1_{Q_{j}} \psi\left(\tau^{2} L\right)^{2}\left(1_{Q_{j}} 1_{Q_{k}} f\right), 1_{Q_{k}} f \rangle \frac{d \tau}{\tau}=0.
$$
The kernel assumption \eqref{eq:kernel} also implies that
\begin{equation}
\label{eq:OD}
\forall p \in (1,\infty) \quad \forall j \in \mathbb{Z}^{d} \quad 
\|1_{Q_{j}}R_{j}1_{Q_{j}}\|_{B(L^{1},L^{\infty})}
\lesssim \rho(x_{j})^{-d}.
\end{equation}
with constant independent of $j$.
Consequently, using $p \leq 2$ and $\rho \in L^{\infty}$, we have that:
\begin{align*}
& \left\|H_{1} \right\|_{\ell^{p}\left(L^{p}\left(L^{2}\right)\right)} \lesssim\left(\sum_{j \in \mathbb{Z}^{d}}\left\|1_{Q_{j}} R_{j}^{\frac{1}{2}} 1_{Q_{j}} \int_{0}^{\rho\left(x_{j}\right)} \psi\left(\tau^{2} L\right) F\left(j, \cdot,\tau\right) \frac{d \tau}{\tau}\right\|_{p}^{p}\right)^{\frac{1}{p}} \\
&\quad \lesssim \left(\sum_{j \in \mathbb{Z}^{d}}
\rho(x_{j})^{d(1-p/2)}
\left\|R_{j}^{\frac{1}{2}} 1_{Q_{j}} \int_{0}^{\rho\left(x_{j}\right)} \psi\left(\tau^{2} L\right) F\left(j, \cdot,\tau\right) \frac{d \tau}{\tau}\right\|_{2}^{p}\right)^{\frac{1}{p}} \\
& \quad \lesssim \left(\sum_{j \in \mathbb{Z}^{d}}\rho(x_{j})^{d(1-p/2)}\left\|R_{j}^{\frac{1}{2}} 1_{Q_{j}}\right\|_{B\left(L^{p}, L^{2}\right)}^{p}\left\|\int_{0}^{\rho\left(x_{j}\right)} \psi\left(\tau^{2} L\right) F\left(j, \cdot,\tau\right) \frac{d \tau}{\tau}\right\|_{p}^{p}\right)^{\frac{1}{p}} \\
&\quad  \lesssim \left(\sum_{j \in \Z^{d}}\rho(x_{j})^{d(1-p/2)}\left\|1_{Q_{j}} R_{j} 1_{Q_{j}}\right\|_{B\left(L^{p}, L^{p^{\prime}}\right)}^{p}\left\|\int_{0}^{\rho\left(x_{j}\right)} \psi\left(\tau^{2} L\right) F\left(j, \cdot,\tau\right)\frac{d \tau}{\tau}\right\|_{p}^{p}\right)^{\frac{1}{p}} \\
&\quad \lesssim \left( \sum_{j \in \mathbb{Z}^{d}}\left\|\left(\int_{0}^{\infty} \left| \int_{0}^{\rho\left(x_{j}\right)} \psi\left(\sigma^{2} L\right) \psi\left(\tau^{2} L\right) F\left(j, \cdot,\tau\right) \frac{d \tau}{\tau}\right|^{2} \frac{d \sigma}{\sigma}\right)^{\frac{1}{2}}\right\|_{p}^{p}\right)^{\frac{1}{p}} \\
& \quad \lesssim\|F\|_{\ell^{p}\left(L^{p}\left(L^{2}\right)\right)},
\end{align*}
where we have used \eqref{eq:rbdSchur} and \eqref{eq:holocalc} in the last line.
The $H_{2}$ term is handled similarly as follows:
\begin{align*}
\left\|H_{2} F\right\|_{\ell^{p}\left(L^{p}\left(L^{2}\right)\right)} &\lesssim\left(\sum_{j \in \mathbb{Z}^{d}}\left\|1_{Q_{j}} R_{j}^{\frac{1}{2}} 1_{Q_{j}} \left(\sum_{k \in \mathbb{Z}^{d}} \int_{R}^{2 R} R_{k}^{\frac{1}{2}} F(k, \cdot, \tau) \frac{d \tau}{\tau}\right)\right\|_{p}^{p}\right)^{\frac{1}{p}} \\
&=\left(\sum_{j \in \mathbb{Z}^{d}}\left\|1_{Q_{j}} R_{j}^{\frac{1}{2}} 1_{Q_{j}} \int_{R}^{2 R} R_{j}^{\frac{1}{2}} F(j, \cdot, \tau) \frac{d \tau}{\tau}\right\|_{p}^{p}\right)^{\frac{1}{p}} \\
&\lesssim\left(\sum_{j \in \mathbb{Z}^{d}}\left\|1_{Q_{j}} R_{j}1_{Q_{j}} \int_{R}^{2 R} F(j, \cdot, \tau) \frac{d \tau}{\tau}\right\|_{p}^{p}\right)^{\frac{1}{p}} \\
& \lesssim \left(\sum_{j \in \Z^{d}}\left\|1_{Q_{j}} R_{j} 1_{Q_{j}}\right\|_{B\left(L^{p}\right)}^{p}\left\|\int_{R}^{2 R} F(j, \cdot, \tau) \frac{d \tau}{\tau}\right\|_{p}^{p}\right)^{\frac{1}{p}} \\
& \lesssim\|F\|_{\ell^{p}\left(L^{p}\left(L^{2}\right)\right)},
\end{align*}
where have used \eqref{eq:OD} in the last step.
We finally turn to $H_{3}$, noting that, by \eqref{eq:finitespeed} and \eqref{eq:localrest} we have that:
\[
\begin{aligned}
\|H_{3}\|_{\ell^{p}\left(L^{p}\left(L^{2}\right)\right)} ^{p}
&= \|(j,\sigma,x) \mapsto 1_{\left[0, \rho\left(x_{j}\right)\right]}(\sigma)  \psi\left(\sigma^{2} L\right)  \int_{R}^{2 R} R_{j}^{\frac{1}{2}} F\left(k, \cdot, \tau\right) \frac{d \tau}{\tau}\|_{\ell^{p}\left(L^{p}\left(L^{2}\right)\right)} ^{p}\\
& \lesssim \sum_{k \in \mathbb{Z}^{d}}\left\|\left(\int_{0}^{\infty}\left|\psi\left(\sigma^{2} L\right) \int_{R}^{2 R} R_{k}^{\frac{1}{2}} F(k, \cdot, \tau) \frac{d \tau}{\tau}\right|^{2} \frac{d \sigma}{\sigma}\right)^{\frac{1}{2}}\right\|_{p}^{p}.
\end{aligned}
\]
We then use \eqref{eq:holocalc}, and the fact that $p \leq 2$, to obtain that
\[
\begin{aligned}
\|H_{3}\|_{\ell^{p}\left(L^{p}\left(L^{2}\right)\right)} ^{p}
& \lesssim 
\sum_{k\in \mathbb{Z}^{d}}\left\|\int_{R}^{2 R} R_{k}^{\frac{1}{2}} F(k, \cdot, \tau) \frac{d \tau}{\tau}\right\|_{p}^{p} \\
& \lesssim 
\sum_{k\in \mathbb{Z}^{d}}\rho(x_{k})^{d(1-p/2)}\left\|\int_{R}^{2 R} R_{k}^{\frac{1}{2}} F(k, \cdot, \tau) \frac{d \tau}{\tau}\right\|_{2}^{p} \\
& \lesssim \sum_{k \in \mathbb{Z}^{d}}\rho(x_{k})^{d(1-p/2)}\left\|R_{k}^{\frac{1}{2}} 1_{Q_{k}}\right\|_{B\left(L^{p}, L^{2}\right)}^{p}\left\|\left(\int_{R}^{2 R}|F(k, \cdot, \tau)|^{2} \frac{d \tau}{\tau}\right)^{\frac{1}{2}}\right\|_{p}^{p} \\
& \lesssim\|F\|_{\ell^{p}\left(L^{p}\left(L^{2}\right)\right)} ^{p},
\end{aligned}
\]
which concludes the proof.

\end{proof}


\end{document}